\documentclass[a4paper,10pt,twoside]{article}

\def\ifundefined#1{\expandafter\ifx\csname#1\endcsname\relax}

\usepackage{graphicx}
\usepackage{graphics}
\usepackage{mathrsfs}
\usepackage{float}
\usepackage{rotating}

\usepackage[T1]{fontenc}
\usepackage[english]{babel}
\usepackage[utf8]{inputenc}
\usepackage{amsthm,amsmath,amssymb,upref,amscd,amsfonts}  
\usepackage{euscript}
\usepackage{epsfig}
\usepackage{ae}
\usepackage{aecompl}
\usepackage[metapost]{mfpic}
\usepackage{calc}
\usepackage{enumerate}
\usepackage{enumitem}
\usepackage{url}
\usepackage[affil-it]{authblk}

\theoremstyle{plain}
\newtheorem{theorem}{Theorem}[section]
\newtheorem{proposition}[theorem]{Proposition}
\newtheorem{lemma}[theorem]{Lemma}
\newtheorem{corollary}[theorem]{Corollary}
\newtheorem{theorem*}{Theorem}

\theoremstyle{definition}

\newtheorem{remark}[theorem]{Remark}

\newlength{\normalparindent}
\numberwithin{equation}{section}

\mathchardef\sa="303A

\DeclareMathOperator*{\esssup}{ess\,sup}
\DeclareMathOperator*{\essinf}{ess\,inf}

\newcommand{\rr}{\ensuremath{s}}
\renewcommand{\epsilon}{\varepsilon}
\newcommand{\R}{\ensuremath{\mathbf{R}}}
\newcommand{\pval}{\ensuremath{\mathrm{p.v. }}} 

\newcommand{\Np}[3][p]{\ensuremath{\mathcal{N}_{#1}(#2 \, ; \, #3)}} 
\newcommand{\Npw}[1][p]{\ensuremath{\mathcal{N}_{#1}}}

\newcommand{\Yonew}{\ensuremath{Y^{1,p}_{M}(\R^N)}}
\newcommand{\Wloc}{\ensuremath{W^{1,p}_{\mathrm{loc}}(\R^N \setminus \{ 0 \})}}
\newcommand{\Lloc}[1][1]{\ensuremath{L^{#1}_{\mathrm{loc}}(\R^N \setminus \{ 0 \})}}
\newcommand{\Xp}[1][1]{\ensuremath{X^p(\R^N)}}

\newcommand{\ld}[1][x]{\ensuremath{\Phi(#1)}}
\newcommand{\vp}{\ensuremath{\varphi}}
\newcommand{\lcw}{\ensuremath{\Lambda}}
\newcommand{\lcwz}{\ensuremath{\Lambda_0}}
\newcommand{\lcwm}{\ensuremath{\Lambda_{*}}}

\renewcommand{\S}{\ensuremath{\mathcal{S}}}

\newcommand{\Nhi}{\ensuremath{M}}

\newcommand{\Ca}{\ensuremath{c_1}}
\newcommand{\Cb}{\ensuremath{c_3}}
\newcommand{\Cc}{\ensuremath{c_2}}

\newcommand{\we}{\ensuremath{\gamma}}
\newcommand{\wt}{\ensuremath{\Gamma}}

\newcommand{\BLg}[1][\wt]{\ensuremath{{BL}^{1,p}_{{#1}}(\R^N)}} 
\newcommand{\Lpg}[1][\we]{\ensuremath{{L}^{p}_{{#1}}(\R^N)}}

\title{Two Weight Estimates for the Single Layer Potential on Lipschitz Surfaces with Small Lipschitz Constant}
\author{Johan Thim ({\small \tt{johan.thim@liu.se}})}
\affil{\small Department of Mathematics, University of Link\"{o}ping, Link\"{o}ping, Sweden}

\date{\today}

\begin{document}

\maketitle
\begin{abstract}
\noindent
This article considers two weight estimates for the single layer potential ---
corresponding to the Laplace operator in~$\R^{N+1}$ ---  
on Lipschitz surfaces with small Lipschitz constant.
We present conditions on the weights to obtain solvability and uniqueness results
in weighted Lebesgue spaces and weighted homogeneous Sobolev spaces, where
the weights are assumed to be radial and doubling. In the case when the
weights are additionally assumed to be differentiable almost everywhere, simplified
conditions in terms of the logarithmic derivative are presented, and as an
application, we prove that the operator corresponding to the single layer potential
in question is an isomorphism between certain weighted spaces of the type mentioned above. 
Furthermore, we consider several explicit weight functions. In particular, we present results
for power exponential weights which generalize known results for
the case when the single layer potential is reduced to a Riesz potential,
which is the case when the Lipschitz surface is given by a hyperplane.
\end{abstract}


\vspace{0.75cm}


\medskip


\section{Introduction}
The {single layer potential} on a Lipschitz surface~$S$ in~$\R^{N+1}$, where~$N \geq 2$, is defined by 
\begin{equation}
\label{eq:simplelayerpot}
\int_{S} \frac{U(Q)}{|P-Q|^{N-1}} \, dS(Q) \text{,} \quad P \in S \text{,}
\end{equation}
where~$dS$ is the Euclidean surface measure. In this article, the surface~$S$ is assumed to be
the graph of a Lipschitz function~$\vp \colon \R^N \rightarrow \R$ such that~$\vp(0) = 0$. 
We parametrize the single layer potential in~(\ref{eq:simplelayerpot}) 
by~$\ld[x] = (x,\vp(x))$ for~$x$ in~$\R^N$ and denote the resulting operator by~$\S$, i.e., formally
\[
\S u(x) = 
\int_{\R^N} \frac{u(y) \sqrt{1 + |\nabla \vp(y)|^2}}{|\ld[x] - \ld[y]|^{N-1}} \, dy \text{,} \quad x \in \R^N \text{.}
\]
The aim of this article is to describe under what conditions the single layer potential operator is an isomorphism between
certain weighted Lebesgue spaces and weighted homogeneous Sobolev spaces. 
The single layer potential is an important object that arises naturally in, e.g., the direct
approach for solving Laplace's equation with boundary integral methods; see for instance Hsiao
and Wendland~\cite{wendland}. There are several other types of applications, and the single layer potential
is a rather well studied object. We refer to Kozlov, Wendland, and Goldberg~\cite{KozlovWendland}, Costabel~\cite{costabel},
and references found therein for applications and properties of the single layer potential.
Moreover, we note that in the case that the surface is the hyperplane~$x_{N+1} = 0$, 
the single layer potential is reduced to a Riesz potential of order one. The Riesz potentials
are named
after Marcel Riesz, who introduced them in the 1930s~\cite{Riesz1939,Riesz1949}. These objects
are well known and can be used in connection with, e.g., fractional integrals; see Rubin~\cite{Rubin1996}.
The results presented in this paper generalize those found
in Kozlov, Thim, and Turesson~\cite{thim1} for Riesz potentials on~$\R^N$. 

We will now focus our attention on investigating the equation
\begin{equation}
\label{eq:maineq}
\S u (x) = f(x) \text{,} \quad x \in \R^N \text{.}
\end{equation}
Specifically, we consider two-weighted estimates for solutions to~(\ref{eq:maineq})
in weighted $L^p$-spaces, with right-hand side in weighted homogeneous Sobolev spaces,
similar to those found in Section~7.5 in~\cite{kozlovmazyaDEOC},
or in Section~8 of~\cite{thim1} for the Riesz potential case with power exponential weights. 
We will rely on results from Kozlov, Thim, and Turesson~\cite{thim3}, where we investigated
the influence of perturbations of a surface like a cone by a small Lipschitz perturbation
and results were expressed in terms of seminorms and 
the function~$\lcw(r)$. This functions is defined as the Lipschitz 
constant of~$\vp$ on a ball of radius~$2r$: 
\begin{equation}
\label{eq:def_lc}
\lcw(r) = \sup_{|x|,|y| \leq 2r, \; x \neq y} \frac{|\vp(x) - \vp(y)|}{|x-y|}, \quad r > 0.
\end{equation}
We denote the global Lipschitz constant of~$\vp$ by~$\lcwz$. Note that we only consider
small perturbations in the sense that~$\lcwz$ is assumed to be sufficiently small, which was
the setting in~\cite{thim3} due to the application of a fixed point theorem in locally convex
spaces~\cite{thim2}. Moreover, in this article, the function~$\lcw$ will be assumed to satisfy 
a Dini-type condition:
\[
\int_0^1 \lcw(\nu) \, \frac{d\nu}{\nu} < \infty.
\]
The class of weights we will consider consist of positive and radial functions on~$(0,\infty)$
that are doubling in the sense
that there exist positive constants~$C_{\we}$ and~$C_{\wt}$ such that
\[
\we(2x) \leq C_{\we} \we(x) \qquad \mbox{and} \qquad
\wt(2x) \leq C_{\wt} \wt(x)
\]
for almost every~$x \in \R^N$. 
A prototypical example of weights of this type is given by power exponentials, e.g.,~$\we(r) = r^{\alpha}$ for~$r > 0$. 
We introduce the space~$\Lpg$ as the Banach space consisting of locally integrable functions
on~$\R^N \setminus \{0\}$ such that
\[
\| u \|_{\Lpg} = \biggl( \int_{\R^N} \we(x)^p |u(x)|^p \, dx \biggr)^{1/p} < \infty.
\]
The space~$\BLg$ is the homogeneous Sobolev space consisting of functions from the local
Sobolev space~$\Wloc$ such that
\[
\| f \|_{\BLg} =  \biggl( \int_{\R^N} \frac{\wt(x)^p}{|x|^p}  |f(x)|^p \, dx 
+ \int_{\R^N} \wt(x)^p |\nabla f(x)|^p \, dx \biggr)^{1/p}\text{.}
\]
One can prove that the first term in the norm of~$\BLg$ can be estimated by the second term in the same norm if
\begin{equation}
\label{eq:i:cond_eq_norm}
\sup_{r > 0} \biggl( \int_0^r s^{N-1-p} \wt(s)^p ds \biggr)^{1/p}
		\biggl( \int_{r}^{\infty} s^{-(N-1)/(p-1)} \wt(s)^{-p'} ds \biggr)^{1/p'}
  < \infty;
\end{equation}
see Lemma~\ref{l:eq_norm}. 
From this it follows that the expression~$\|\nabla f\|_{\Lpg[\wt]}$ defines an equivalent norm on~$\BLg$
if~(\ref{eq:i:cond_eq_norm}) holds. 
Conditions of the type seen in~(\ref{eq:i:cond_eq_norm}) are frequently present in this article. These types
of conditions, sometimes referred to as conditions of Muckenhoupt type, are sufficient and necessary to obtain
weighted Hardy inequalities. We refer to Muckenhoupt~\cite{Muckenhoupt1972} and Maz'ya~\cite{mazya_sobolev},
and references found therein. 

The main results in this article are the following two theorems concerning existence and uniqueness
for~(\ref{eq:maineq}) when the spaces~$\Lpg$ and~$\BLg$ are considered. The results are rather technical, and for that reason,
we present in Theorem~\ref{t:Lisomorphism} an application where the weights are assumed to be differentiable almost everywhere with respect to the radial
coordinate. This simplifies the conditions significantly.
The existence and uniqueness results follow from corresponding theorems for local spaces derived previously in Theorems~1.1 and~1.2 of~\cite{thim3}.
Indeed, Theorem~1.1 of~\cite{thim3} states that there exist positive constants\/~$\lcwm$,\/~$\Ca$, and\/~$\Cc$, depending only on\/~$N$
and\/~$p$, such that if\/~$\lcwz \leq \lcwm$, we obtain existence and uniqueness results for the equation in~(\ref{eq:maineq})
under certain restrictions on~$p$ and the involved functions. We will use these constants throughout this article, and 
in accordance with the proof of Lemma~3.7 in~\cite{thim3}, we also assume that~$\Ca \lcwm \leq 1/2$ and that~$\Cc \lcwm \leq (N-1)/2$. 
We use the notation~$M = N - \Cc \lcwz$, and moreover, tacitly assume that~$1 < p < \infty$ if nothing else is stated.
The conjugate exponent~$p'$ is defined as~$p' = p/(p-1)$ throughout. 

\begin{theorem}
\label{t:i:exist}
Suppose that~{\rm(}\ref{eq:i:cond_eq_norm}{\rm)} holds, 
\begin{equation}
\label{eq:i:cond_J1}
\sup_{r > 0} 
	\biggl( \int_0^{r} s^{p'(M-N/p) - 1} \wt(s)^{-p'} \, ds\biggr)^{1/p'} 
	\biggl( \int_{r}^{\infty} s^{N-1-Mp} \we(s)^p \, ds \biggr)^{1/p}
< \infty
\end{equation}
and
\begin{equation}
\label{eq:i:cond_J2}
\begin{aligned}
\sup_{r > 0} {} & \biggl( \int_0^{r} s^{N-1} \we(s)^p 
		\exp\biggl( -\Ca p \int_0^{s} \lcw(\nu) \, \frac{d\nu}{\nu} \biggr)
		\, ds \biggr)^{1/p}\\
	& \qquad \cdot \biggl( \int_r^{\infty} s^{-1 - Np'/p} \wt(s)^{-p'} 
			\exp\biggl( \Ca p' \int_0^{s} \lcw(\nu) \, \frac{d\nu}{\nu} \biggr)
			ds
		\biggr)^{1/p'} < \infty,
\end{aligned}
\end{equation}
where~$N/M < p < N$. Then, if\/~$f \in \BLg$,
the equation in~{\rm(}\ref{eq:maineq}{\rm)} has a solution\/~$u \in \Lpg$ such that
\[
\| u \|_{\Lpg}  \leq  C \| f \|_{\BLg},
\]
where the constant~$C$ depends on~$N$,~$p$, the doubling constants of~$\we$ and~$\wt$, and the supremums above.
\end{theorem}

\noindent
We wish to remark here that in the case when~$\we = \wt$, the condition in~(\ref{eq:i:cond_eq_norm}) implies that~(\ref{eq:i:cond_J2}) is true.

\begin{theorem}
\label{t:i:uniq}
Suppose that\/~$u \in \Lpg$, where\/~$N/M < p < N$, and that
\begin{equation}
\label{eq:i:uniq_zero}
\we^{-1}(r) = O(r^{N/p - N + \Cc \lcwz}), \quad \mbox{as } r \rightarrow 0,
\end{equation}
and
\begin{equation}
\label{eq:i:uniq_infty}
\we^{-1}(r) = O \biggl( r^{N/p} 
		\exp\biggl( -\Ca \int_1^{r} \lcw(\nu) \, \frac{d\nu}{\nu} \biggr) \biggr), 
		\quad \mbox{as } r \rightarrow \infty.
\end{equation}
If\/~$\S u = 0$, then it follows that\/~$u = 0$.
\end{theorem}

\noindent We prove Theorems~\ref{t:i:exist} and~\ref{t:i:uniq} in Section~\ref{s:mainiso},
and as an application of these theorems,
we prove in Section~\ref{s:isomorphism} the following isomorphism result.

\begin{theorem}
\label{t:Lisomorphism}
If~$\we$ and~$\wt$ are 
functions differentiable almost everywhere such that
\begin{equation}
\label{eq:total_lip_we}
\Ca \lcwz - \frac{N}{p} < \essinf_{r > 0} \frac{r\we'(r)}{\we(r)} \leq \esssup_{r > 0} \frac{r\we'(r)}{\we(r)} < N - \Cc\lcwz - \frac{N}{p},
\end{equation}
\begin{equation}
\label{eq:total_lip_wt}
1 - \frac{N}{p} < \essinf_{r > 0} \frac{r\wt'(r)}{\wt(r)} \leq \esssup_{r > 0} \frac{r\wt'(r)}{\wt(r)} < N - \Cc\lcwz - \frac{N}{p},
\end{equation}
and 
\begin{equation}
\label{eq:total_lip_wewt}
\max \left\{ \esssup_{r > 0}  \frac{\wt(r)}{\we(r)}, \; 
\esssup_{r > 0} \frac{\we(r)}{\wt(r)} \right\} < \infty,
\end{equation}
where~$N/M < p < N$,
then the operator~$\S$ is an isomorphism between~$\Lpg$ and $\BLg$.
\end{theorem}

\noindent Note that~(\ref{eq:total_lip_wt}) implies~(\ref{eq:i:cond_eq_norm}); see Corollary~\ref{c:eq_norm}. 
Moreover, if~$\we = \wt$, then only~(\ref{eq:total_lip_wt}) is necessary out of~(\ref{eq:total_lip_we})--(\ref{eq:total_lip_wewt}) 
for Theorem~\ref{t:Lisomorphism} to hold. 

In Section~\ref{s:applications}, we consider some explicit examples of weights. In particular, we present an 
application of Theorem~\ref{t:Lisomorphism} where the weights are equal
and given by a power exponential function, i.e.,~$\we(r) = \wt(r) = r^{\alpha}$, 
where~$\alpha \in \R$ satisfies  
\begin{equation}
\label{eq:i:alpha}
1 < \alpha + \frac{N}{p} < N - \Ca \lcwz.
\end{equation}
Theorem~\ref{t:power} states that~$\S$ is an isomorphism between~$\Lpg$ and~$\BLg$ when~$\alpha$ satisfies~(\ref{eq:i:alpha}).
We specifically note that with~$\alpha = 0$ (which satisfies the condition above), we obtain an
isomorphism between~$L^p(\R^N)$ and the homogeneous Sobolev space~$\BLg$ (with~$\wt = 1$).
Moreover, we also note that these results reduce to the corresponding results for Riesz potentials in
the case when~$\lcwz = 0$, i.e., the hyperplane case~$x_{N+1} = 0$; we refer to Section~8 in~\cite{thim1}.

Throughout this paper, the constant~$C$ is a generic constant that can change from line to line, but only depends
on the parameters, e.g.,~$N$,~$p$, and the weights~$\we$ and~$\wt$. 


\section{Preliminary Results}
\label{s:prelim}

\subsection{Single Layer Potentials on Lipschitz Surfaces}
Let us recall some properties of single layer potentials on Lipschitz surfaces that were discussed
in~\cite{thim3}. The case when~$x_{N+1} = 0$, where the single layer potential becomes a Riesz potential, 
is covered in~\cite{thim1}.

In our previous analysis, results were formulated in terms of the family of seminorms defined by
\[
\Np{u}{r} = \biggl( \frac{1}{r^N} 
	\int_{r \leq |x| < 2r} |u(x)|^p \, dx \biggr)^{1/p}
\text{,} \quad r > 0  \text{.}
\]
The Banach space~$\Xp[1]$ consists of all functions~$u \in \Lloc[p]$ such that
\begin{equation}
\label{eq:defXp}
\int_{0}^{1} \rr^N \, \Np{u}{\rr} \, \frac{d\rr}{\rr} 
+
\int_{1}^{\infty} \Np{u}{\rr} \, d\rr 
< \infty,
\end{equation}
where the left-hand side defines the norm on~$\Xp$.
This space is the natural domain, in terms of the seminorms~$\Npw$, 
for the operator~$\S$ in the case that~$S$ is the hyperplane~$x_{N+1} = 0$; 
this is discussed further in~\cite{thim1}. 
For~$0 < \Nhi \leq N$, 
the Banach space~$\Yonew$ consists of all functions~$f$ in~$\Wloc$ such that
\begin{equation}
\label{eq:Yonew}
\int_0^{1} \rr^{\Nhi} \, \Np{\nabla f}{\rr} \, \frac{d\rr}{\rr}
+ \int_1^{\infty} \Np{\nabla f}{\rr} \, d\rr < \infty
\end{equation}
and~$\lim_{r \rightarrow \infty} \int_{S^{N-1}} f(r\theta) \, dS(\theta) = 0$, where~$S^{N-1}$ is the unit sphere in~$\R^N$
and~$dS$ is the Euclidian surface measure.
The left-hand side of~(\ref{eq:Yonew}) defines the norm on this space. 
We note that~(\ref{eq:Yonew}) implies that the limit in the definition exists.
The {\it raison d'\^{e}tre} for the space~$\Yonew$ is that if~$M = N - \Cc \lcwz$, then solutions to~(\ref{eq:maineq}) 
exist; see Theorem~\ref{t:exist} below. 
For a comparison of these spaces with weighted Lebesgue spaces and weighted Sobolev spaces, 
see Lemma~\ref{l:inclusions}.

The single layer potential $\S u$ is weakly differentiable if~$u \in \Xp$ and
\begin{equation}
\label{eq:diff_Su}
\partial_k \S u(x) = (1-N)\bigl( T_k u(x) + \partial_k \vp (x)T_{N+1} u(x) \bigr)
\end{equation}
for $x \in \R^N$ and $k = 1, 2, \ldots, N$, where~$T_k$ are the singular integral operators 
defined by 
\[
T_{k} u(x) = \pval \int_{\R^N} 
\frac{(\ld[x] - \ld[y])_k}{|\ld[x] - \ld[y]|^{N+1}} \, u(y) \sqrt{1 + |\nabla \vp(y)|^2} \, dy
\text{,} \quad x \in \R^N \text{,}
\]
for~$k=1,2,\ldots,N+1$. Here,~$( \ld[x] )_k$ denotes the $k$th component of the vector~$\ld[x]$. 
In Section~2.2 of~\cite{thim3}, it is shown that if
$u \in \Lloc[p]$, $1 < p < \infty$, satisfies
\begin{equation}
\label{eq:defTk}
\int_{0}^{1} \rr^N \, \Np{u}{\rr} \, \frac{d\rr}{\rr} 
+
\int_{1}^{\infty} \Np{u}{\rr} \, \frac{d\rr}{\rr}  < \infty,
\end{equation}
then~$T_k u$ is defined almost everywhere and
\begin{equation}
\label{eq:est_Tk}
\Np{T_k u}{r} \leq C \int_0^{r} \left( \frac{\rr}{r} \right)^N \, \Np{u}{\rr} \, \frac{d\rr}{\rr} 
+ C \int_r^{\infty}  \Np{u}{\rr} \, \frac{d\rr}{\rr} 
\quad \mbox{for } r > 0,
\end{equation}
for~$k=1,2,\ldots,N+1$, where~$C$ only depends on~$N$ and~$p$.

Furthermore, in~\cite{thim3}, the following existence and uniqueness results were proved.

\begin{theorem}
\label{t:exist}
There exist positive constants\/~$\lcwm$,\/~$\Ca$,\/~$\Cc$, and\/~$\Cb$, depending only on\/~$N$
and\/~$p$, such that if\/~$\lcwz \leq \lcwm$ and 
if\/~$f \in \Yonew$ with\/~$M = N - \Cc \lcwz$ and~\/$1 < p < \infty$,
then~{\rm(}\ref{eq:maineq}{\rm)} has a solution\/~$u$ in~$\Xp$.
For\/~$r > 0$, this solution satisfies
\begin{equation}
\label{eq:est_sol}
\begin{aligned}
\Np{u}{r} \leq {} & \Cb \int_0^{r} \left( \frac{\rr}{r} \right)^{\Nhi} \Np{\nabla f}{\rr} \, \frac{d\rr}{\rr}\\
& + \Cb \int_r^{\infty} \exp \biggl( 
\Ca \int_r^{\rr} \lcw(\nu) \, \frac{d\nu}{\nu}
\biggr) \, \Np{\nabla f}{\rr} \, \frac{d\rr}{\rr}.
\end{aligned}
\end{equation}
\end{theorem}

\begin{theorem}
\label{t:uniq}
Suppose that\/~$u \in \Lloc[p]$, where\/~$1 < p < \infty$, and that~$u$ satisfies~{\rm(}\ref{eq:defXp}{\rm)},
\begin{equation}
\label{eq:est_uniq_infty}
\Np{u}{r} = 
O\biggl( \exp \biggl( -\Ca \int_1^r \lcw(\nu) \, \frac{d\nu}{\nu} \biggr) \biggr)   \quad \text{as } r \rightarrow \infty \text{,}
\end{equation}
and
\begin{equation}
\label{eq:est_uniq_zero}
\Np{u}{r} = 
O\bigl( r^{-\Nhi} \bigr)  \quad \text{as } r \rightarrow 0 \text{,}
\end{equation}
where\/~$\Ca$,\/~$\Cc$,\/~$\Nhi$, and\/~$\lcwm$ are as in Theorem~\ref{t:exist}, and\/~$\lcwz \leq \lcwm$.
If\/~$\S u = 0$, then it follows that\/~$u = 0$.
\end{theorem}

\noindent We would like to point out that the solution in Theorem~\ref{t:exist} satisfies the conditions in Theorem~\ref{t:uniq}; 
see Remark~3.12 in~\cite{thim3}. 

\subsection{Weighted Inequalities}

Before proving the main results, we collect some technical lemmas.
To shift between the spaces~$\Xp$ and~$\Yonew$ and their weighted counterparts,
we will employ Lemma~2.1 in~\cite{thim1} which we now present to assist the reader.

\begin{lemma}
\label{l:wL1eqX1}
Suppose that $0 \leq M_1 < M_2 \leq \infty$. Then, for~$u \in L^1_{\rm{loc}}(\R^N \setminus \{0\})$,
\[
\begin{aligned}
 \int_{M_1 \leq |x| < M_2} \frac{|u(x)|}{|x|^{N-1}} \, dx \leq 
	C \int_{M_1/2}^{M_2} \Np[1]{u}{\rho} \, d\rho
	\leq \int_{{M_1}/{2} \leq |x| < 2M_2} \frac{|u(x)|}{|x|^{N-1}} \, dx \text{,}
\end{aligned}
\]
where~$C = (N-1)/(2^{N-1}-1)$.
\end{lemma}

\noindent
We also present the weighted Hardy inequalities we will rely on. The proof can be found in, e.g.,
Muckenhoupt~\cite{Muckenhoupt1972}.

\begin{lemma}
\label{l:hardy}
Suppose that~$g$,~$U$, and~$V$ are measurable functions on~$(0,\infty)$. Then
\begin{enumerate}
\item[{\rm (i)}] there exists a constant~$C_1$ such that
\begin{equation}
\label{eq:hardy1}
\biggl( \int_{0}^{\infty} \biggl| U(r) \int_{0}^{r} g(s) ds \biggr|^p dr \biggr)^{1/p} 
\leq
C_1 \biggl( \int_0^{\infty} |V(r) g(r)|^p dr \biggr)^{1/p}
\end{equation}
if and only if
\[
B_1 = \sup_{r > 0} \biggl( \int_{r}^{\infty} |U(s)|^p ds \biggr)^{1/p} \biggl( \int_{0}^{r} |V(s)|^{-p'} ds \biggr)^{1/p'} < \infty,
\]
\item[{\rm (ii)}]
and there exists a constant~$C_2$ such that
\begin{equation}
\label{eq:hardy2}
\biggl( \int_{0}^{\infty} \biggl| U(r) \int_{r}^{\infty} g(s) ds \biggr|^p dr \biggr)^{1/p} 
\leq
C_2 \biggl( \int_0^{\infty} |V(r) g(r)|^p dr \biggr)^{1/p}
\end{equation}
if and only if
\[
B_2 = \sup_{r > 0} \biggl( \int_{0}^{r} |U(s)|^p ds \biggr)^{1/p} \biggl( \int_{r}^{\infty} |V(s)|^{-p'} ds \biggr)^{1/p'} < \infty.
\]
\end{enumerate}
\noindent
The least constants~$C_j$ for which~{\rm(}\ref{eq:hardy1}{\rm)} and~{\rm(}\ref{eq:hardy2}{\rm)} hold satisfies
\[
B_j \leq C_j \leq p^{1/p}(p')^{1/p'} B_j \quad \mbox{for } j=1,2.
\]
\end{lemma}

\noindent The results of Lemma~\ref{l:hardy} also hold for~$p=1$ and~$p=\infty$ with the natural conventions.
However, we will only use~$1 < p < \infty$ in this paper.

If the weights are assumed to be differentiable, we can often reduce the complexity 
of the derived formulas by the following lemma.

\begin{lemma}
\label{l:PI}
Let~$g$ be a non-negative function on~$(0,\infty)$ that is almost everywhere differentiable, 
and suppose that~$\alpha$ is a positive constant and that~$\beta$ is a real nonzero constant. Then
\begin{enumerate}
\item[\rm (i)] if
\begin{equation}
\label{eq:PI_cond_zero}
C_{*} = \essinf_{0 < s < r} \biggl( 1 + \frac{\beta}{\alpha} \frac{s g'(s)}{g(s)} \biggr) > 0,
\end{equation}
then 
\begin{equation}
\label{eq:PI_final_zero}
\int_0^r s^{\alpha - 1} g(s)^{\beta} \, ds \leq \frac{1}{\alpha C_{*}} \, r^{\alpha} g(r)^{\beta}
\quad \mbox{for } r > 0 \mbox{\rm ;}
\end{equation}
\item[\rm (ii)] if 
\begin{equation}
\label{eq:PI_cond_infty}
C^{*} = \essinf_{s > r} \biggl( 1 - \frac{\beta}{\alpha} \frac{sg'(s)}{g(s)} \biggr) > 0,
\end{equation}
then
\begin{equation}
\label{eq:PI_final_infty}
\int_r^{\infty} s^{-\alpha - 1} g(s)^{\beta} \, ds \leq \frac{1}{\alpha C^{*}} \, r^{-\alpha} g(r)^{\beta}
\quad \mbox{for } r > 0 \mbox{\rm .}
\end{equation}
\end{enumerate}
\end{lemma}

\begin{proof}
To prove~(i), we use integration by parts and obtain that
\[
\int_0^r s^{\alpha - 1} g(s)^{\beta} \, ds \leq \frac{r^{\alpha}g(r)^{\beta}}{\alpha} - \frac{\beta}{\alpha} \int_0^r s^{\alpha} g(s)^{\beta - 1} g'(s) \, ds, 
\]
or equivalently, that
\[
\int_0^r s^{\alpha - 1} g(s)^{\beta} \biggl( 1 + \frac{\beta}{\alpha} \frac{sg'(s)}{g(s)}  \biggr) \, ds \leq \frac{r^{\alpha}g(r)^{\beta}}{\alpha}.
\]
Using~(\ref{eq:PI_cond_zero}), we obtain the estimate in~(\ref{eq:PI_final_zero}).

Similarly, using integration by parts we also obtain that
\[
\int_r^{\infty} s^{-\alpha - 1} g(s)^{\beta} \, ds \leq \frac{r^{-\alpha}g(r)^{\beta}}{\alpha} + 
								\frac{\beta}{\alpha} \int_0^r s^{-\alpha} g(s)^{\beta - 1} g'(s) \, ds, 
\]
or equivalently, that
\[
\int_r^{\infty} s^{\alpha - 1} g(s)^{\beta} \biggl( 1 - \frac{\beta}{\alpha} \frac{sg'(s)}{g(s)}  \biggr) \, ds \leq \frac{r^{-\alpha}g(r)^{\beta}}{\alpha}.
\]
Using~(\ref{eq:PI_cond_infty}), we obtain the estimate in~(\ref{eq:PI_final_infty}).
\end{proof}

\section{Main Results}

We now proceed to prove the main results of this paper. 

\subsection{Weighted Spaces}

We start by proving when there is an equivalent norm on~$\BLg$, which will simplify some of our calculations.

\begin{lemma}
\label{l:eq_norm}
Suppose that
\begin{equation}
\label{eq:cond_eq_norm}
B = \sup_{r > 0} \biggl( \int_0^r s^{N-1-p} \wt(s)^p ds \biggr)^{1/p}
		\biggl( \int_{r}^{\infty} s^{-(N-1)/(p-1)} \wt(s)^{-p'} ds \biggr)^{1/p'}
  < \infty.
\end{equation}
Then the expression~$\| \nabla f \|_{\Lpg[\wt]}$ defines an equivalent norm on~$\BLg$.
\end{lemma}

\begin{proof}
Since~$\lim_{r \rightarrow \infty} f(r\omega) = 0$ for all~$\omega \in S^{N-1}$,
polar coordinates and a weighted Hardy inequality (Lemma~\ref{l:hardy}) 
implies that
\[
\begin{aligned}
\int_{\R^N} \frac{|f(x)|^p \, \wt(x)^p}{|x|^p} \, dx 
= {} & 
\int_{S^{N-1}} \int_{0}^{\infty} r^{N-1-p} |f(r\omega)|^p \wt(r)^p  dr \, d\omega\\
\leq {} & 
\int_{S^{N-1}} \int_{0}^{\infty} \biggl| 
	r^{-1+(N-1)/p} \wt(r) \int_{r}^{\infty} f'_{s}(s\omega) ds  
	\biggr|^p  dr \, d\omega\\
\leq {} & 
C \int_{S^{N-1}} \int_{0}^{\infty} \biggl| 
	r^{(N-1)/p} \wt(r) f'_{s}(s\omega)  
	\biggr|^p  dr \, d\omega\\
\leq {} & 
C \int_{\R^{N}}  \wt(x)^p |\nabla f(x)|^p \, dx \\
\end{aligned}
\]
is true if and only if (\ref{eq:cond_eq_norm}) holds. Here,~$f'_s$ denotes the radial derivative of~$f$
and the constant~$C$ depends on~$B$,~$N$, and~$p$.
\end{proof}

\begin{corollary}
\label{c:eq_norm}
If the weight\/~$\wt$ is differentiable almost everywhere, the condition in~{\rm(}\ref{eq:cond_eq_norm}{\rm)} of 
Lemma~\ref{l:eq_norm} can be replaced by~$1 < p < N$ and
\begin{equation}
\label{eq:Lcond_eq_norm}
\inf_{r > 0} \frac{r\wt'(r)}{\wt(r)}  > 1 - \frac{N}{p}.
\end{equation}
\end{corollary}

\begin{proof}
This result follows from two applications of Lemma~\ref{l:PI} with~$g(s) = \wt(s)$. First, we put~$\alpha = N - p$ and~$\beta = p$. 
The condition that~$\alpha > 0$ can be expressed as~$p < N$, and~(\ref{eq:PI_cond_zero}) is given by~(\ref{eq:Lcond_eq_norm}).
Secondly, we put~$\alpha = (N-p)/(p-1)$ and~$\beta = -p'$.
The condition that~$\alpha > 0$ can again be expressed as~$p < N$, and~(\ref{eq:PI_cond_infty}) is reduced to~(\ref{eq:Lcond_eq_norm}).
Thus, by Lemma~\ref{l:PI}(i) and~(ii),
\[
B \leq \sup_{r > 0} C r^{(N-p)/p} \wt(r) r^{-(N-p)/p} \wt(r)^{-1} = C < \infty. \qedhere
\]
\end{proof}

\noindent
We now have sufficient tools to investigate when the weighted spaces~$\Lpg$ and~$\BLg$ are included in~$\Xp$ and~$\Yonew$, respectively. 

\begin{lemma}
\label{l:inclusions}
The following inclusions are valid.
\begin{enumerate}
\item[{\rm(i)}]
If
\begin{equation}
\label{eq:cond_LpgXp}
\int_0^1 s^{N} \we(s)^{-p'} \, \frac{ds}{s} 
+
\int_1^{\infty} s^{-Np'/p} \we(s)^{-p'} \, \frac{ds}{s} < \infty,
\end{equation}
then~$\Lpg \subset \Xp$. 
\item[{\rm(ii)}]
If\/~{\rm(}\ref{eq:i:cond_eq_norm}{\rm)} holds and
\begin{equation}
\label{eq:cond_BLgYp}
\int_0^1 s^{(M - N/p)p'}\wt(s)^{-p'} \, \frac{ds}{s} +
\int_1^{\infty} s^{-Np'/p}\wt(s)^{-p'} \, \frac{ds}{s} < \infty,
\end{equation}
then~$\BLg \subset \Yonew$. 
\end{enumerate}
\end{lemma}

\begin{proof}
H\"older's inequality and Lemma~\ref{l:wL1eqX1} imply that
\[
\int_0^1 s^{N} \Np{u}{s} \, \frac{ds}{s}
\leq 
\biggl( \int_0^1 s^{(N - N/p)p'}\we(s)^{-p'} \, \frac{ds}{s} \biggr)^{1/p'}
\biggl( \int_{|x| \leq 2} \we^p |u|^p \biggr)^{1/p}
\]
and
\[
\int_1^{\infty} s \, \Np{u}{s} \, \frac{ds}{s}
\leq 
\biggl( \int_1^{\infty} s^{-Np'/p} \we(s)^{-p'} \, \frac{ds}{s} \biggr)^{1/p'}
\biggl( \int_{|x| > 1} \we^p |u|^p \biggr)^{1/p}.
\] 
Thus, it is sufficient that~(\ref{eq:cond_LpgXp}) holds for the inclusion to be true.

Similarly, by H\"older's inequality and Lemma~\ref{l:wL1eqX1},
\[
\int_0^1 s^M \, \Np{\nabla f}{s} \, \frac{ds}{s}
\leq 
\biggl( \int_0^1 s^{(M - N/p)p'}\wt(s)^{-p'} \, \frac{ds}{s} \biggr)^{1/p'}
\biggl( \int_{|x| \leq 2} \wt^p|\nabla f|^p \biggr)^{1/p}
\]
and
\[
\int_1^{\infty} s \, \Np{\nabla f}{s} \, \frac{ds}{s}
\leq 
\biggl( \int_1^{\infty} s^{-Np'/p}\wt(s)^{-p'} \, \frac{ds}{s} \biggr)^{1/p'}
\biggl( \int_{|x| > 1} \wt^p |\nabla f|^p \biggr)^{1/p}.
\] 
Thus, it is sufficient that~(\ref{eq:cond_BLgYp}) and~(\ref{eq:i:cond_eq_norm}) hold for us to obtain that~$\BLg$
is a subset of~$\Yonew$. 
\end{proof}

\begin{corollary}
\label{c:inclusions}
If\/~$\we$ and\/~$\wt$ are differentiable almost everywhere, 
the condition in~{\rm(}\ref{eq:cond_LpgXp}{\rm)} can be replaced by 
\begin{equation}
\label{eq:Lcond_LpgXp}
\esssup_{0 < r < 1} \frac{r\we'(r)}{\we(r)} < N - \frac{N}{p}
\qquad \mbox{and} \qquad
\essinf_{r > 1} \frac{r\we'(r)}{\we(r)} > -\frac{N}{p},
\end{equation}
while the condition in~{\rm(}\ref{eq:cond_BLgYp}{\rm)} can be replaced by 
\begin{equation}
\label{eq:Lcond_BLgYp}
\esssup_{0 < r < 1} \frac{r\wt'(r)}{\wt(r)} < M - \frac{N}{p},
\qquad \mbox{and} \qquad
\essinf_{r > 0} \frac{r\wt'(r)}{\wt(r)} > 1 - \frac{N}{p},
\end{equation}
if~$Mp > N$.
\end{corollary}

\begin{proof}
Similarly to the proof of Corollary~\ref{c:eq_norm}, we obtain the desired
result from Lemma~\ref{l:PI}.
Indeed, Lemma~\ref{l:PI}(i) with~$\alpha = N$,~$\beta = -p'$, and~$g(s) = \we(s)$, shows that
the first integral in the left-hand side of~(\ref{eq:cond_LpgXp}) is finite
if~(\ref{eq:Lcond_LpgXp}) holds. 
Lemma~\ref{l:PI}(ii) with~$\alpha = Np'/p$,~$\beta = -p'$, and~$g(s) = \we(s)$, shows that
the second integral in the left-hand side of~(\ref{eq:cond_LpgXp}) is finite
if~(\ref{eq:Lcond_LpgXp}) holds. 
Lemma~\ref{l:PI}(i) with~$\alpha = (M - N/p)p'$,~$\beta = -p'$, and~$g(s) = \wt(s)$, shows that
the first integral in the left-hand side of~(\ref{eq:cond_BLgYp}) is finite
if~(\ref{eq:Lcond_BLgYp}) holds and~$Mp > N$. 
Lemma~\ref{l:PI}(ii) with~$\alpha = Np'/p$,~$\beta = -p'$, and~$g(s) = \wt(s)$, shows that
the second integral in the left-hand side of~(\ref{eq:cond_BLgYp}) is finite
if~(\ref{eq:Lcond_BLgYp}) holds. 
Note also that~(\ref{eq:Lcond_BLgYp}) implies~(\ref{eq:i:cond_eq_norm}).
\end{proof}

\subsection{Continuity of~$\S$}

It is clear from Section~\ref{s:prelim} that~$\S$ maps~$\Xp$ into~$\Wloc$.
However, we need conditions for when the operator~$\S$ is bounded as a mapping from~$\Lpg$ into~$\BLg$.

\begin{proposition}
Suppose that
\label{p:S_cont}
\begin{equation}
\label{eq:cond_SJ1}
B_1 = \sup_{r > 0} 
		\biggl( \int_0^{r} s^{N-1} \we(s)^{-p'} \, ds\biggr)^{1/p'} 
		\biggl( \int_{r}^{\infty} s^{N-1-Np} \wt(s)^p \, ds \biggr)^{1/p} 
 < \infty
\end{equation}
and
\begin{equation}
\label{eq:cond_SJ2}
B_2 = \sup_{r > 0} \biggl( \int_0^{r} s^{N-1} \wt(s)^p \, ds \biggr)^{1/p}
	 \biggl( \int_r^{\infty} s^{-(1 + Np'/p)} \we(s)^{-p'} ds \biggr)^{1/p'} < \infty.
\end{equation}
Then
\[
\| \nabla \S u \|_{\Lpg[\wt]} \leq C \| u \|_{\Lpg}, 
\]
where~$C$ depends only on~$N$,~$p$,~$B_1$,~$B_2$, and the doubling constants.
If also~{\rm(}\ref{eq:i:cond_eq_norm}{\rm)} holds,
then~$\S$ is a continuous mapping from~$\Lpg$ into~$\BLg$. 
\end{proposition}

\begin{proof}
Since~$u \in \Xp$ by Lemma~\ref{l:inclusions}, we obtain by (\ref{eq:diff_Su}) and~(\ref{eq:est_Tk}) that
\begin{equation}
\label{eq:est_diff_Su}
\Np{\nabla \S u}{r} \leq C \int_0^{r} \left( \frac{r}{s} \right)^N \, \Np{u}{s} \, \frac{ds}{s}
+ C \int_r^{\infty}  \Np{u}{s} \, \frac{ds}{s}, \quad r > 0.
\end{equation}
Thus, by Lemma~\ref{l:wL1eqX1} and~(\ref{eq:est_diff_Su}),
\begin{equation}
\label{eq:est_Su_Bvp}
\begin{aligned}
\| \S u \|_{\Lpg[\wt]}^p = {} & 
\int_{\R^N} \wt(x)^p |\nabla \S u(x)|^{p} \, dx\\
\leq {} &
C \int_{0}^{\infty} \Np[1]{|x|^{N-1}\wt(x)^p|\nabla \S u(x)|^{p}}{r} \, dr\\
\leq {} &
C \int_{0}^{\infty} r^{N-1} \wt(r)^p \Np{\nabla \S u}{r}^p \, dr\\
\leq {} &
C \int_{0}^{\infty} r^{N-1} \wt(r)^p \left( 
	\int_{0}^{r} \left( \frac{s}{r} \right)^{N} \Np{u}{s} \, \frac{ds}{s}
\right)^p \, dr\\
& +
C \int_{0}^{\infty} r^{N-1} \wt(r)^p \left( 
	\int_{r}^{\infty} \Np{u}{s} \, \frac{ds}{s}
\right)^p \, dr,
\end{aligned}
\end{equation}
where the constant~$C$ depends on~$N$,~$p$, and~$C_{\wt}$.
We denote the integrals on the right-hand side by~$J_1$ and~$J_2$, respectively,
and prove that~$J_1$ and~$J_2$ can be estimated by~$\| u \|_{\Lpg}^p$.
Let us consider~$J_1$ first. Then there exists a constant~$C > 0$ such that
\[
\begin{aligned}
J_1 = {} & \int_0^{\infty} \biggl| 
	r^{(N-1)/p - N} \wt(r) \int_0^{r} s^{N-1} \, \Np{u}{s} \, ds \biggr|^p \, dr\\
\leq {} & 
	C \int_0^{\infty} \bigl|
		r^{(N-1)/p} \we(r) \, \Np{u}{r} 
	\bigr|^p \, dr
\end{aligned}
\]
if and only if~(\ref{eq:cond_SJ1}) holds (Lemma~\ref{l:hardy}), and then~$J_1$ satisfies
\begin{equation}
\label{eq:final_bound_SJ1}
\begin{aligned}
J_1 &\leq	C \int_0^{\infty} 
		r^{N-1} \we(r)^{p} \, \Np[1]{|u|^p}{r} \, dr\\
& \leq 
	C \int_0^{\infty} 
		r^{N-1} \, \Np[1]{\we^p \, |u|^p}{r} \, dr\\
& \leq 
	C \int_{\R^N} \we(x)^p \, |u(x)|^p \, dx, \\
\end{aligned}
\end{equation}
where we used Lemma~\ref{l:wL1eqX1} and the doubling condition of~$\we$.
The constant~$C$ depends on~$B_1$,~$N$,~$p$, and the doubling constants of~$\we$ and~$\wt$.

Similarly, there exists a constant~$C > 0$ such that
\[
\begin{aligned}
J_2 = {} & \int_0^{\infty} \biggl| 
	r^{(N-1)/p - N} \wt(r) \int_{r}^{\infty}
		 \Np{u}{s} \, \frac{ds}{s} \biggr|^p \, dr\\
\leq {} & 
	C \int_0^{\infty} \bigl|
		r^{(N-1)/p} \we(r) \, \Np{u}{r} 
	\bigr|^p \, dr
\end{aligned}
\]
if and only if~(\ref{eq:cond_SJ2}) holds, and then analogously with~(\ref{eq:final_bound_SJ1}),~$J_2$ satisfies
\[
\begin{aligned}
J_2 \leq C \int_{\R^N} \we(x)^p \, |u(x)|^p \, dx,
\end{aligned}
\]
where~$C$ depends on~$B_2$,~$N$,~$p$, and the doubling constants of~$\we$ and~$\wt$.
\end{proof}

\begin{remark}
\label{r:Xp}
Note that~(\ref{eq:cond_SJ1}) and~(\ref{eq:cond_SJ2}) imply~(\ref{eq:cond_LpgXp}).
This proves that~$\Lpg$ is a subset of~$\Xp$ if the conditions in Proposition~\ref{p:S_cont} are satisfied.
Moreover, if~$\we = \wt$, then~(\ref{eq:i:cond_eq_norm}) implies~(\ref{eq:cond_SJ2}). 
\end{remark}

\begin{corollary}
\label{c:S_cont}
If~$\we$ and~$\wt$ are differentiable almost everywhere, 
the conditions in~{\rm(}\ref{eq:cond_SJ1}{\rm)} and~{\rm(}\ref{eq:cond_SJ2}{\rm)} 
of Proposition~\ref{p:S_cont} can be replaced by 
\begin{equation}
\label{eq:Lcond_Scont1a}
\max \left\{ 
\esssup_{r > 0} \frac{r\we'(r)}{\we(r)} , \; 
\esssup_{r > 0} \frac{r\wt'(r)}{\wt(r)} \right\} < N - \frac{N}{p}, 
\end{equation}
\begin{equation}
\label{eq:Lcond_Scont1b}
\min \left\{ 
\essinf_{r > 0} \frac{r\we'(r)}{\we(r)} , \; 
\essinf_{r > 0} \frac{r\wt'(r)}{\wt(r)}  \right\} > -\frac{N}{p} ,
\end{equation}
and
\begin{equation}
\label{eq:Lcond_Scont2}
\esssup_{r > 0} \frac{\wt(r)}{\we(r)} < \infty.
\end{equation}
\end{corollary}

\begin{proof}
Similarly with the proof of Corollary~\ref{c:eq_norm},
we obtain this Corollary by means of Lemma~\ref{l:PI}. 
Specifically, we use Lemma~\ref{l:PI}(i) with~$\alpha = N$,~$\beta = -p'$, and~$g(s) = \we(s)$,
and Lemma~\ref{l:PI}(ii) with~$\alpha = Np - N$,~$\beta = p$, and~$g(s) = \wt(s)$.
This proves that~(\ref{eq:Lcond_Scont1a}) is sufficient 
for~(\ref{eq:PI_final_zero}) and~(\ref{eq:PI_final_infty}), which in turn proves that
the left-hand side of~(\ref{eq:cond_SJ1}) is finite if~(\ref{eq:Lcond_Scont2}) holds since
\[
B_1 \leq \sup_{r > 0} C r^{N/p'} \we(r)^{-1} r^{-N + N/p} \wt(r) = 
C \sup_{r > 0} \frac{\wt(r)}{\we(r)}. 
\]
Secondly,
Lemma~\ref{l:PI}(i) with~$\alpha = N$,~$\beta = p$, and~$g(s) = \wt(s)$,
and Lemma~\ref{l:PI}(ii) with~$\alpha = Np'/p$,~$\beta = -p'$, and~$g(s) = \we(s)$,
proves that 
(\ref{eq:Lcond_Scont1b}) is sufficient 
for~(\ref{eq:PI_final_zero}) and~(\ref{eq:PI_final_infty}), and hence, 
the left-hand side of~(\ref{eq:cond_SJ2}) is finite if~(\ref{eq:Lcond_Scont2}) holds since
\[
B_2 \leq \sup_{r > 0} C r^{N/p} \wt(r) r^{-N/p} \we(r)^{-1} = 
C \sup_{r > 0} \frac{\wt(r)}{\we(r)}. 
\qedhere
\]
\end{proof}

\subsection{Existence and Uniqueness of Solutions}
Since we consider equation~(\ref{eq:maineq}) for functions
in the weighted spaces $\Lpg$ and~$\BLg$, we need to relate
these spaces to Theorems~\ref{t:exist} and~\ref{t:uniq}.
The following proposition provides conditions on the weights
to obtain a class in which solutions to~(\ref{eq:maineq}) are unique.

\begin{proposition}
\label{p:uniq}
Suppose that~$u \in \Lpg$, where~$\we$ satisfies
\[
\we^{-1}(r) = O(r^{N/p - M}), \quad \mbox{as } r \rightarrow 0
\]
and
\[
\we^{-1}(r) = O \biggl( r^{N/p} 
		\exp\biggl( -\Ca \int_1^{r} \lcw(\nu) \, \frac{d\nu}{\nu} \biggr) \biggr), 
		\quad \mbox{as } r \rightarrow \infty.
\]
Then~$u \in \Xp$ and~$u$ satisfies~{\rm(}\ref{eq:est_uniq_infty}{\rm)} 
and~{\rm(}\ref{eq:est_uniq_zero}{\rm)} in Theorem~\ref{t:uniq}.
\end{proposition}

\begin{proof}
The fact that~$u \in \Xp$ is a consequence of Lemma~\ref{l:inclusions}.
Indeed, the conditions required in Proposition~\ref{p:uniq} imply that both integrals
in~(\ref{eq:cond_LpgXp}) are finite, as can be verified directly.
Moreover, we know that~$u \in \Lpg$, which implies that for~$r > 0$,
\[
\begin{aligned}
\Np{u}{r}^p = {} & r^{-N} \int_{r \leq |x| < 2r} \we(x)^{-p} \we(x)^p |u(x)|^p dx \\  
\leq {} & C r^{-N} \we(r)^{-p} \int_{\R^N} \we(x)^{p} |u(x)|^p  dx\\
= {} & C r^{-N} \we(r)^{-p} \| u \|_{\Lpg}^p,
\end{aligned}
\]
where the constant~$C$ depends on the doubling constant of~$\we$. 
Thus it is clear that the the corresponding conditions
in Theorem~\ref{t:uniq} hold. 
\end{proof}

By Theorem~\ref{t:exist}, we now that~(\ref{eq:maineq}) has a solution~$u$ if~$f$
satisfies certain conditions. For this solution, the estimate in~(\ref{eq:est_sol}) is valid.
We now prove that for functions~$u \in \Xp$ that satisfies this estimate, the 
mapping~$f \mapsto u$ is a bounded operator from~$\BLg$ into~$\Lpg$ when the weights
are sufficiently nice.

\begin{proposition}
\label{p:inverse_cont}
Suppose that~$u \in \Xp$ satisfies~{\rm(}\ref{eq:est_sol}{\rm)} in Theorem~\ref{t:exist}.
If
\begin{equation}
\label{eq:cond_J1}
B_1 = \sup_{r > 0} 
	\biggl( \int_0^{r} s^{p'(M-N/p) - 1} \wt(s)^{-p'} \, ds\biggr)^{1/p'} 
	\biggl( \int_{r}^{\infty} s^{N-1-Mp} \we(s)^p \, ds \biggr)^{1/p}
< \infty
\end{equation}
and
\begin{equation}
\label{eq:cond_J2}
\begin{aligned}
B_2 = \sup_{r > 0} {} & \biggl( \int_0^{r} s^{N-1} \we(s)^p 
		\exp\biggl( -\Ca p \int_0^{s} \lcw(\nu) \, \frac{d\nu}{\nu} \biggr)
		\, ds \biggr)^{1/p}\\
	& \qquad \cdot \biggl( \int_r^{\infty} s^{-1 - Np'/p} \wt(s)^{-p'} 
			\exp\biggl( \Ca p' \int_0^{s} \lcw(\nu) \, \frac{d\nu}{\nu} \biggr)
			ds
		\biggr)^{1/p'} < \infty,
\end{aligned}
\end{equation}
then
\[
\| u \|_{\Lpg}  \leq C \| \nabla f \|_{\Lpg[\wt]}, 
\]
where the constant~$C$ depends on~$B_1$,~$B_2$,~$N$,~$p$, and the doubling constants.
\end{proposition}

\begin{proof}
Similarly with the proof of Proposition~\ref{p:S_cont},
\begin{equation}
\label{eq:est_u_Lpg}
\begin{aligned}
\| u \|_{\Lpg}^p = {} & 
\int_{\R^N} \we(x)^p |u(x)|^{p} \, dx\\
\leq {} &
C \int_{0}^{\infty} \Np[1]{|x|^{N-1}\we(x)^{p}|u(x)|^{p}}{r} \, dr\\
\leq {} &
C \int_{0}^{\infty} r^{N-1} \we(r)^p \Np{u}{r}^p \, dr,
\end{aligned}
\end{equation}
where the constant~$C$ depends on~$C_{\we}$ and~$N$.
Now,~(\ref{eq:est_sol}) implies that the right-hand side of~(\ref{eq:est_u_Lpg}) is bounded by
\begin{equation}
\label{eq:sork1}
\begin{aligned}
&
C \int_{0}^{\infty} r^{N-1} \we(r)^p \left( 
	\int_{0}^{r} \left( \frac{s}{r} \right)^{M} \Np{\nabla f}{s} \, \frac{ds}{s}
\right)^p \, dr\\
& \quad +
C \int_{0}^{\infty} r^{N-1} \we(r)^p \left( 
	\int_{r}^{\infty} \exp\biggl( \Ca \int_{r}^s \lcw(\nu) \, \frac{d\nu}{\nu} \biggr) \, \Np{\nabla f}{s} \, \frac{ds}{s}
\right)^p \, dr.
\end{aligned}
\end{equation}
We denote the integrals in~(\ref{eq:sork1}) by~$J_1$ and~$J_2$, respectively,
and prove that both integrals can be bounded by~$\| \nabla f \|_{\Lpg[\we]}^p$.
Let us consider~$J_1$ first. Then there exists a constant~$C > 0$ such that
\[
\begin{aligned}
J_1 = {} & \int_0^{\infty} \biggl| 
	r^{(N-1)/p - M} \we(r) \int_0^{r} s^{M-1} \, \Np{\nabla f}{s} \, ds \biggr|^p \, dr\\
\leq {} & 
	C \int_0^{\infty} \bigl|
		r^{(N-1)/p} \wt(r) \, \Np{\nabla f}{r} 
	\bigr|^p \, dr
\end{aligned}
\]
if and only if~(\ref{eq:cond_J1}) holds (Lemma~\ref{l:hardy}), and then~$J_1$ satisfies
\begin{equation}
\label{eq:final_bound_J1}
\begin{aligned}
J_1 &\leq	C \int_0^{\infty} 
		r^{N-1} \wt(r)^{p} \, \Np[1]{|\nabla f|^p}{r} \, dr\\
& \leq 
	C \int_0^{\infty} 
		r^{N-1} \, \Np[1]{\wt^p \, |\nabla f|^p}{r} \, dr\\
& \leq 
	C \int_{\R^N} \wt(x)^p \, |\nabla f(x)|^p \, dx \\
\end{aligned}
\end{equation}
where we used Lemma~\ref{l:wL1eqX1} and the doubling condition of~$\wt$.
The constant~$C$ now depends on~$B_1$,~$N$,~$p$, and the doubling constants. 

Similarly, there exists a constant~$C > 0$ such that
\[
\begin{aligned}
J_2 = {} & \int_0^{\infty} \biggl| 
	r^{(N-1)/p} \we(r) 
	\exp\biggl( -\Ca \int_0^{r} \lcw(\nu) \, \frac{d\nu}{\nu} \biggr)\\ 
	& \quad \cdot \int_{r}^{\infty}  
		\exp\biggl( \Ca \int_0^{s} \lcw(\nu) \, \frac{d\nu}{\nu} \biggr) \,
		 \Np{\nabla f}{s} \, \frac{ds}{s} \biggr|^p \, dr\\
\leq {} & 
	C \int_0^{\infty} \bigl|
		r^{(N-1)/p} \wt(r) \, \Np{\nabla f}{r} 
	\bigr|^p \, dr
\end{aligned}
\]
if and only if~(\ref{eq:cond_J2}) holds, and then analogously with~(\ref{eq:final_bound_J1}),~$J_2$ satisfies
\[
\begin{aligned}
J_2 \leq C \int_{\R^N} \wt(x)^p \, |\nabla f(x)|^p \, dx,
\end{aligned}
\]
where~$C$ depends on~$B_2$,~$N$,~$p$, and the doubling constants. 
\end{proof}

\begin{remark}
\label{r:Yonew}
Note that~(\ref{eq:cond_J1}) and~(\ref{eq:cond_J2}) imply~(\ref{eq:cond_BLgYp}).
This proves that the space~$\BLg$ is a subset of~$\Yonew$ if the conditions in Proposition~\ref{p:inverse_cont} are satisfied
and~(\ref{eq:i:cond_eq_norm}) holds.
\end{remark}

\begin{corollary}
\label{c:Linverse_cont}
If~$\we$ and~$\wt$ are differentiable almost everywhere, 
the conditions in~{\rm(}\ref{eq:cond_J1}{\rm)} and~{\rm(}\ref{eq:cond_J2}{\rm)} 
of Proposition~\ref{p:inverse_cont} can be replaced by~$Mp > N$,  
\begin{equation}
\label{eq:Lcond_inverse_cont1a}
\max \left\{ 
\esssup_{r > 0} \frac{r\we'(r)}{\we(r)} , \; 
\esssup_{r > 0} \frac{r\wt'(r)}{\wt(r)} \right\} < M - \frac{N}{p},
\end{equation}
\begin{equation}
\label{eq:Lcond_inverse_cont1b}
\min \left\{ 
\essinf_{r > 0} \biggl( \frac{r\we'(r)}{\we(r)} - \Ca \lcw(r) \biggr) , \; 
\essinf_{r > 0} \biggl( \frac{r\wt'(r)}{\wt(r)} - \Ca \lcw(r) \biggr) \right\} > -\frac{N}{p} ,
\end{equation}
and
\begin{equation}
\label{eq:Lcond_inverse_cont2}
\esssup_{r > 0} \frac{\we(r)}{\wt(r)} < \infty.
\end{equation}

\end{corollary}

\begin{proof}
We proceed as in the proof of Corollary~\ref{c:eq_norm}.
Indeed, Lemma~\ref{l:PI}(i) with~$\alpha = p'(M-N/p)$,~$\beta = -p'$, and~$g(s) = \wt(s)$,
and Lemma~\ref{l:PI}(ii) with~$\alpha = Mp - N$,~$\beta = p$, and~$g(s) = \we(s)$,
proves that 
(\ref{eq:Lcond_inverse_cont1a}) is sufficient 
for~(\ref{eq:PI_final_zero}) and~(\ref{eq:PI_final_infty}), which proves that
the left-hand side of~(\ref{eq:cond_J1}) is finite if~(\ref{eq:Lcond_inverse_cont2}) holds since
\[
B_1 \leq \sup_{r > 0} C r^{M - N/p} \wt(r)^{-1} r^{-M + N/p} \we(r) = 
\sup_{r > 0} \frac{\we(r)}{\wt(r)} < \infty. 
\]
Secondly, Lemma~\ref{l:PI}(i) with~$\alpha = N$,~$\beta = p$, and
\[
g(s) = \exp \biggl( \log \we(s) - \Ca \int_0^s \lcw(\nu) \, \frac{d\nu}{\nu} \biggr), \quad s > 0,
\]
and Lemma~\ref{l:PI}(ii) with~$\alpha = Np'/p$,~$\beta = -p'$, and
\[
g(s) = \exp \biggl( \log \wt(s) - \Ca \int_0^{s} \lcw(\nu) \, \frac{d\nu}{\nu} \biggr), \quad s > 0,
\]
proves that 
(\ref{eq:Lcond_inverse_cont1b}) is sufficient for~(\ref{eq:PI_final_zero}) and~(\ref{eq:PI_final_infty}), which in turn
shows that
the left-hand side of~(\ref{eq:cond_J2}) is finite if~(\ref{eq:Lcond_inverse_cont2}) holds since
\[
\begin{aligned}
B_2 \leq {} & \sup_{r > 0} C r^{N/p} \we(r) 
		\exp \biggl( -\Ca \int_0^r \lcw(\nu) \, \frac{d\nu}{\nu} \biggr) 
	r^{-N/p} \wt(r)^{-1} 
		\exp \biggl( \Ca \int_0^r \lcw(\nu) \, \frac{d\nu}{\nu} \biggr) \\
= {} & \sup_{r > 0} \frac{\we(r)}{\wt(r)} < \infty. \qedhere
\end{aligned}
\]
\end{proof}

\subsection{Proof of Theorems~\ref{t:i:exist} and~\ref{t:i:uniq}}
\label{s:mainiso}
The proof of our main results is basically the ``sum'' of the propositions in the preceding
sections. We start by proving Theorem~\ref{t:i:exist}.
It is clear from Remark~\ref{r:Yonew} that~(\ref{eq:i:cond_J1}) and~(\ref{eq:i:cond_J2}) imply
that~$f \in \Yonew$ if~(\ref{eq:i:cond_eq_norm}) holds, and thus, Theorem~\ref{t:exist} proves 
that there exist a solution~$u \in \Xp$
to~(\ref{eq:maineq}) which satisfies~(\ref{eq:est_sol}). Proposition~\ref{p:inverse_cont} now concludes
the proof of Theorem~\ref{t:i:exist}. 

Turning our attention to Theorem~\ref{t:i:uniq}, we note that~(\ref{eq:i:uniq_zero}) and~(\ref{eq:i:uniq_infty}) are the conditions in
Proposition~\ref{p:uniq}, which proves that solutions in~$\Lpg$ which satisfy
(\ref{eq:i:uniq_zero}) and~(\ref{eq:i:uniq_infty}) are unique by Theorem~\ref{t:uniq}. 

\subsection{Proof of Theorem~\ref{t:Lisomorphism}}
\label{s:isomorphism}
Similarly with Section~\ref{s:mainiso},
we first need to require that~$\Lpg$ and~$\BLg$ are subsets of~$\Xp$
and~$\Yonew$, respectively. However, this is true if the conditions
of Propositions~\ref{p:S_cont} and~\ref{p:inverse_cont} are satisfied 
(see Remarks~\ref{r:Xp} and~\ref{r:Yonew}), and
since~(\ref{eq:total_lip_we}),~(\ref{eq:total_lip_wt}), and~(\ref{eq:total_lip_wewt}), are 
sufficient for Corollaries~\ref{c:eq_norm},~\ref{c:S_cont}, and~\ref{c:Linverse_cont}, we have the
necessary requirements for both the inclusions and moreover, the continuity of both~$\S$ and a
presumptive inverse mapping satisfying~(\ref{eq:est_sol}).
Furthermore, since~$f \in \BLg$ implies that~$f \in \Yonew$ in our case, 
Theorem~\ref{t:exist} provides the existence of a solution~$u$
to~$\S u = f$. Next we need to make sure that solutions are unique so that an inverse 
exists for the spaces in question.
It is true that, for almost every~$s > 0$,
\[
\Ca \lcwz -\frac{N}{2} < \frac{s \we'(s)}{\we(s)}
\quad \Leftrightarrow \quad  
\frac{1}{s} \biggl( \Ca \lcwz - \frac{N}{2} \biggr) < \frac{d}{ds} \bigl( \log \we(s) \bigr). 
\]
Integrating both sides, we obtain that
\[
\log r^{\Ca \lcwz - N/2} < \int_1^r \frac{d}{ds} \bigl( \log \we(s) \bigr) \, ds = \log \we(r) - \log \we(1), \quad \mbox{for } r > 0, 
\]
or equivalently,
\[
\we(r) > \we(1) r^{\Ca \lcwz - N/2}, 
\quad \mbox{for } r > 0.
\]
Thus,
\[
\frac{\we(r)^{-1}}{r^{N/p - M}} \leq C r^{N/2} 
\mbox{  as } r \rightarrow 0,
\]
if~$N - \Ca \lcwz - \Cc \lcwz > 0$, which is true if~$\Ca \lcwz \leq 1/2$ and~$\Cc \lcwz \leq (N-1)/2$,
and
\[
\frac{\we(r)^{-1}}{r^{N/p - \Ca\lcwz}} \leq C      
\mbox{  as } r \rightarrow \infty.
\]
Hence, the conditions in Proposition~\ref{p:uniq} are satisfied, and~$\S$ is injective viewed as a mapping
from~$\Lpg$ into~$\BLg$.

This concludes the proof of Theorem~\ref{t:Lisomorphism}. Indeed, for the assumptions
in the theorem, we have proved that the operator~$\S \colon \Lpg \rightarrow \BLg$ is continuous and 
that the operator~$\S^{-1} \colon \BLg \rightarrow \Lpg$
exists and is also continuous.


\section{Explicit Examples}
\label{s:applications}
\label{s:power}

Let us consider some explicit examples of differentiable weights, 
starting with power exponential weights.

\begin{theorem}
\label{t:power}
Suppose that~$\we(r) = \wt(r) = r^{\alpha}$, where~$\alpha \in \R$. 
If~$N / M < p < N$ and
\begin{equation}
\label{eq:alpha}
1 < \alpha + \frac{N}{p} < N - \Cc\lcwz ,
\end{equation}
then~$\S$ is an isomorphism between~$\Lpg$ and~$\BLg$.
\end{theorem}

\noindent 
Obviously~$\we$ and~$\wt$ are differentiable functions, so Theorem~\ref{t:Lisomorphism} is applicable. 
We thus obtain the requirement~(\ref{eq:alpha}) since~$r\we'(r)\we(r)^{-1} = \alpha$. 
We note here that if~$\lcwz = 0$, meaning that we consider a hyperplane in~$\R^{N+1}$,
we obtain the same result that was presented in Theorem~8.4 of~\cite{thim1} for Riesz potentials. 

An obvious modification is 
given by letting~$\we(r) = r^{\alpha_1}$ when~$0 < r < 1$ and~$\we(r) = r^{\alpha_2}$ when~$r \geq 1$.
Then, if
\[
1 - \frac{N}{p} < \min \{ \alpha_1, \; \alpha_2 \} 
\leq
\max \{ \alpha_1, \; \alpha_2 \} < N - \Cc \lcwz - \frac{N}{p},
\]
the operator~$\S$ is an isomorphism between~$\Lpg$ and~$\BLg$. This allows for separate treatment of behavior 
close to zero and for large arguments.
Furthermore, we can also consider power logarithmic weights. Indeed, 
if~$\we(r) = \wt(r) = \bigl( \log(1 + r) \bigr)^{\alpha}$, where~$\alpha \in \R$,
$N / M < p < N$, and~{\rm(}\ref{eq:alpha}{\rm)} holds,
then~$\S$ is an isomorphism between~$\Lpg$ and~$\BLg$.
This result is not surprising considering that~$\log(1 + r) \approx r$ for small~$r$. More specifically,
it is clear that~$\we$ and~$\wt$ are differentiable functions, so Theorem~\ref{t:Lisomorphism} 
is applicable and
if~$\alpha > 0$, then~$r\we'(r)\we(r)^{-1}$ is a decreasing function such that
\[
0 = \inf_{r > 0} \frac{r\we'(r)}{\we(r)} \leq \sup_{r > 0} \frac{r\we'(r)}{\we(r)} = \alpha,
\]
and if~$\alpha < 0$, then~$r\we'(r)\we(r)^{-1}$ is an increasing function such that
\[
\alpha = \inf_{r > 0} \frac{r\we'(r)}{\we(r)} \leq \sup_{r > 0} \frac{r\we'(r)}{\we(r)} = 0,
\]
which leaves~(\ref{eq:alpha}) intact.


\def\bibname{References}

\end{document}